\documentclass[12pt]{article}

\usepackage[bookmarks,bookmarksnumbered,%
    citebordercolor={0.8 0.8 0.8},filebordercolor={0.8 0.8 0.8},%
    linkbordercolor={0.8 0.8 0.8},%
    urlbordercolor={0.8 0.8 0.8},%
    pagebackref,linktocpage%
    ]{hyperref}

\usepackage[margin=1.0in]{geometry}
\usepackage{graphicx}              
\usepackage{amsmath}               
\usepackage{amsfonts}              
\usepackage{amsthm}                
\usepackage{url}
\usepackage{amssymb}
\usepackage{latexsym}
\usepackage{bm}
\usepackage{enumitem}

\allowdisplaybreaks

\newtheorem{theorem}{Theorem}[section]
\newtheorem{thm}[theorem]{Theorem}
\newtheorem{prop}[theorem]{Proposition}
\newtheorem{lemma}[theorem]{Lemma}

\theoremstyle{definition}

\let\oldmarginpar\marginpar
\renewcommand\marginpar[1]{\-\oldmarginpar[\raggedleft\footnotesize #1]%
{\raggedright\footnotesize #1}}

\def\keywords#1{\bigskip \par\noindent{\it Keywords and phrases: }#1\par}
\def\AMS#1{\par\noindent{\it 2010 Mathematics Subject Classification: }#1\par}

\DeclareMathOperator{\R}{\mathbb{R}}

\DeclareMathOperator{\T}{\mathbb{T}}

\DeclareMathOperator{\Z}{\mathbb{Z}}

\DeclareMathOperator{\eps}{\epsilon}

\DeclareMathOperator{\supp}{\text{supp}}
\DeclareMathOperator{\ra}{\rightarrow}
\DeclareMathOperator{\half}{\frac{1}{2}}
\DeclareMathOperator{\M}{\mathcal{M}}

\title{A differentiation theorem for uniform measures}
\author{Marc Carnovale}
\date{\today}
\begin{document}

\maketitle


\begin{abstract}
Using the notion of higher-order Fourier dimension introduced in \cite{M2} (which was a sort of psuedorandomness condition stemming from the Gowers norms of Additive Combinatorics), 
we prove a maximal theorem and corresponding differentiation theorem for singular measures on $\R^d$, $d=1,2,\dots$. This extends results begun by Hardy and Littlewood for balls in $\R^d$
and continued by Stein \cite{stein} for spheres in $\R^{d\geq 3}$ and Bourgain for circles in $\R^2$, first considered for more general spaces in \cite{rubio}, and shown to hold for some
singular subsets of the reals for the first time in \cite{LabaDiff}. 

Notably, unlike the more delicate of the previous results on differentiation such as \cite{Bourgain} and \cite{LabaDiff}, the assumption of higher-order Fourier dimension subsumes all of
the geometric or combinatorial input necessary for one to obtain our theorem, and suggests a new approach to some problems in Harmonic Analysis.

\end{abstract}

\tableofcontents

\section*{Acknowledgments}

Many thanks to Prof.'s Izabella Laba and Malabika Pramanik for introducing me to the question of arithmetic progressions in fractional sets and differentiation problems
which motivated this work, 
for the sharing of their expertise in Harmonic Analysis, for funding my master's degree, and much more. 

Thanks to Nishant Chandgotia for his forgiveness of the mess I made of our office for two years and his bountiful friendship.

Thanks to Ed Kroc, Vince Chan, and Kyle Hambrook for the frequent use of their time and ears and their ubiquitous encouragement.


\noindent \keywords{Gowers norms, singular measures, uniformity norms, Salem sets, Hausdorff dimension, Fourier dimension, maximal estimates, differentiation theorems}
\vskip0.2in

\noindent \AMS{28A78, 42A32, 42A38, 42A45, 11B25, 26A24, 26A99, 42B25, 28A15}



\section{Introduction}
One of the significant result of 20th Century Harmonic Analysis were Stein's Spherical Maximal Theorem \cite{stein} and Bourgain's Circular Maximal Theorem \cite{Bourgain}, as well as their resulting 
differentiation theorems, which together state that if one takes
a sphere in $\R^d$, $d\geq 2$, translates it to the point $x$, scales it to have radius $r$, and averages the $L^p$ (for certain $p$) function $f$ over it, then as $r$ shrinks to zero, one recovers $f(x)$ almost
everywhere. Bourgain's result required additional geometric input and estimates beyond the Fourier/$L^2$ methods of Stein. Later, Rubio de Francia found a general 
differentiation theorem for measures with sufficient Fourier decay extending Steins, but not Bourgains, result, but the requirements of this theorem were too stringent to say anything about
whether a fractal on the real line could differentiate $L^p$ for any $p$. In 2011, \cite{LabaDiff} constructed sets and measures on the real line, of Hausdorff dimension strictly less than 1,
which nevertheless satisfy a maximal theorem and differentiation theorem, however no satisfying general theorem was avaialable. 

In the present work, we present a maximal and a differentiatin theorem for measures on $\R^d$, $d=1,2,\dots$, under an condition which generalizes Rubio de Francia's Fourier decay assumption.
This condition is that the measure possesses a higher-order Fourier dimension sufficiently close to that of the ambient space, 
which is a sort of psuedorandomness condition stemming from the Gowers norms of Additive Combinatorics. 

Inspired by Gowers' proof of Szemeredi's theorem, in \cite{M1} and \cite{M2} we developed a theory of Gowers uniformity norms for singular measures on the torus and a notion of higher
order Fourier dimension for such measures.

The main definitions were the following.

Given a measure $\mu$ on $\T^d$, we define the measure $\triangle^k\mu$ on $\T^{(k+1)d}$ iteratively by setting $\triangle^0\mu=\mu$ and 

\begin{align}
 \int f \,d\triangle^k\mu = \lim_{n\ra\infty} \int \Phi_{n}\ast\triangle^{k-1}\mu(x-u_k;u')\,d\triangle^{k-1}\mu(x;u')\,du_k
\end{align}

We then defined the $U^k$ norm of $\mu$ and showed it to be equivalent to
\begin{align}
 \|\mu\|_{U^k} = |\triangle^k\mu(\T^{(k+1)d})|^{\frac{1}{2^k}}
\end{align}
and introduced the $k$th order Fourier dimension of $\mu$ as be the supremum over all $\beta\in(0,d)$ for which

\begin{align}\label{intermsof1}
 |\widehat{\triangle^j\mu}(0;\eta)|\leq C_F (1+|\eta|)^{-(j+1)\frac{\beta}{2}}\hspace{15pt} \forall\hspace{3pt} 1\leq j \leq k
\end{align}

One obtains bounds on the $U^k$ norm in terms of (\ref{intermsof1}), and in fact much stronger information. Let $\phi_n$ be some approximate identity on $\R^d$ (for terminology, see the introduction ot \cite{M1}.
Set $\mu_n=\phi_n\ast\mu$. In \cite{M2}, we proved (a statement obviously equivalent to via the triangle inequality) the following

\begin{prop}[Proposition 2 of \cite{M2}]\label{thm:rk}
 Let $\mu$ be a finite compactly supported (Radon) measure on $\R^d$ with a higher order Fourier decay given by (\ref{intermsof1}). Then setting

 \begin{align*}
  r_k:= \bigg(\prod_{j=3}^k \left[2-\frac{{2^{3j-2}}}{2^{3j-2}-[1-\frac{(j+1)\beta}{jd}]}\right]\bigg) (2\beta-d) \end{align*}

we have the bound

\begin{align}
 \|\mu_{n+1}-\mu_n\|_{U^k}\leq C 2^{-\frac{r_k}{2^k}n}
\end{align}
where the constant depends only on the choice of $\phi_n$ and the  constant $C_F$.

\end{prop}

In particular, note that $r_k=r_k(\beta)$ increases as $\beta$ increases and is positive for $\beta$ close enough to $d$, as one would expect.

In this paper, we show that the structural control afforded by a high $k$th order Fourier dimension can be used to yield the conclusion that such measures differentiate $L^p$ for sufficiently large $p$. 
This allows differentiation theorems somewhat akin to Rubio de Francia's \cite{rubio} with lower demands on the dimension (but in general this is a higher-order rather than classical Fourier dimension).
One sense in which this is an improvement on Rubio de Francia's result is that we obtain differentiation theorems for certain objects of (classical) Fourier dimensions exactly equal to $1$. 

We believe that this work clarifies results in \cite{LabaDiff}, where it was first shown that sets and measures of fractional dimension in $\R^1$ may differentiate $L^p$.

Our approach has strong parallels both to \cite{Bourgain} and \cite{LabaDiff}, but puts the analogs of the bounds on ``internal tangencies'' and ``transverse intersections''
behind each of those directly into a framework of multilinear estimates, which ultimately rely only on nothing other than (higher-order) Fourier decay. We control the ``internal tangency'' portion (see (\ref{Ommeggga})) of the argument by combining a universal argument counting the size of translation parameters (cf. Lemma \ref{thm:internals})
which can be close together along with a trivial bound on the size of $|\mu_n|$ coming from the Hausdorff dimension of $\mu$ ( it is technically convenient, though it turns out not necessary, to assume that $\mu$ obeys the ball conditional
\begin{align}\label{ballcondition}
 \mu(B(x,r)) \leq C_H r^{\alpha}
\end{align}
for some $\alpha\in (0,d)$)

The more interesting ``transverse intersection'' portion (see (\ref{omeggga})) of the argument can be phrased as a count of certain linear patterns weighted by $\mu$, which is precisely the sort of thing which Gowers
norms were introduced to control. Here our (higher order) Fourier dimension assumption and the work done in \cite{M2} comes in, providing the needed estimates. (see Lemma \ref{thm:transverse}).

In more detail, we show  that measures of $k-1$-st order Fourier dimension $\beta$ slightly larger than $\frac{k-1}{k}d$ and satisfying (\ref{ballcondition}) for large enough 
$\alpha$ differentiate $L^p(\R^d)$ for $p > \frac{k}{k-1}$, in the sense that for $f\in L^p$, $\lim_{r\ra\infty} \int f(x+ry)\,d\mu(y)\ra f(x)$ for a.e. $x$. 
To show this, we follow what is now a standard approach to reduce the problem to a more congenial one. 

Define $\tilde{\mathcal{M}} f(x) := \sup_{t>0} \int f(x+ty) \, d\mu(y)$. 
In order to prove the differentiation result, it is enough to show the following theorem.

\begin{thm}\label{thm:differential}
 Let $\mu$ be a measure. Then there is a $\beta_0\in(0,d)$ so that if the $k-1$-st order Fourier dimension of $\mu$ is greater than  $\beta_0$, then $\tilde{\mathcal{M}}$ is bounded on $L^p$ for $p>\frac{k}{k-1}$.
\end{thm}


Theorem \ref{thm:differential} will follow from Theorem \ref{thm:sooomany} at the end of this section, since standard arguments allow us to replace the supremum over $t>0$ with a supremum over the single scale $1\leq t\leq2$  and so work with  $\mathcal{M}f(x) := \sup_{t\in(1,2)} \int f(x+ty)\,d\mu(y)$. For completeness, we include these arguments in Section \ref{ch:scaling} below.

Choose $\phi$ a radially symmetric Schwartz function of integral 1  which is equal to the identity in a neighborhood of $0$, set $\phi_N (x) = 2^{-N} \phi(2^{-N}x)$ and 
define $\psi_N$ so that $\psi_N = \phi_{N+1}-\phi_{N}$. 
Set $\mathcal{M}_n  : = \sup_{t\in[1,2]} \int f(x+ty) \,d\mu_n(y)$ where $\mu_n = \phi_n\ast\mu$.

It is clear that $\mathcal{M}$ is bounded if both  $\mathcal{M}_0: f\mapsto \sup_{1\leq t \leq2} \int f(\cdot +ty) \,d\mu_0(t)$, with $\mu_0 = \psi_1\ast\mu$, and $\sum_{n>0} \mathcal{M}_n$ are  bounded; that the former is bounded is  a straightforward consequence of the differentiation theorem of Hardy and Littlewood.  So we study the operators $\mathcal{M}_n$.

\section{A multilinear estimate and uniformity norm}

Our proof of Theorem \ref{thm:d11} (in which we prove the meat of the main theorem) invokes the following lemma in order to utilize our assumptions on the $k$th order Fourier dimension of
the measure $\mu$.

First, we must recall the following facts.

The operator $\triangle^k : L^{\infty}(\R^d)\mapsto L^{\infty} (\R^{(k+1)d})$ is given by
\begin{align}
 \triangle^k f (x;u) = \prod_{\iota\in\{0,1\}^k} f(x-\iota\cdot u) = \triangle^{k-1}f(x-u_k;u')\triangle^{k-1}f(x;u')
\end{align}

In \cite{M1}, we showed that  when $\mu$ is a function $f$, the $U^k$ norm of $\mu$ is the same as $\|f\|_{U^k} = \left(\int \triangle^k f(x;u)\,dx\,du\right)^{\frac{1}{2^k}}$.

For vectors $v$ and $u$, define $vu := (v_1u_1,\dots,v_du_d)$. 

We then have the following lemma.

\begin{lemma}\label{thm:transverse}
Let $f_i$, $i=0,\dots,k$ be bounded functions and $b_i\in \R^d$ distinct vectors with $|b_i-b_j|\geq 1$ for $i\neq j$. Then 
\begin{align}
& \int_{[0,1]^d}\int_{[0,1]^d} \prod_{i=0}^{k} f_i(x-b_ir)\,dx\,dr\\\leq&
 \left(\prod_{i=0}^{k-1}\|f_i\|_{\infty}\right)\|f_k\|_{U^{k+1}}
\end{align}
\end{lemma}
\begin{proof}

For $u\in (\R^d)^n$, define the operator $B_{j}u =( (b_j-b_0)u_1,\cdots,(b_j-b_n)u_n)$.

We claim that 

\begin{align}\label{twobt}
& \int\int \prod_{i=j}^{k} \triangle^j f(x-(b_i-b_{j-1})r;B_iu)\\\leq&
 \|\triangle^j f_j\|_2 [\int\int \prod_{i=j+1}^{k} \triangle^{j+1}f(x-(b_i-b_{j})r;B_iu)]^{\half}
\end{align}
from which the lemma follows via induction since  $\|\triangle^j f_j\|_{2}\leq \|f_j\|_{\infty}^{2^j}$.

To obtain (\ref{twobt}), we apply Cauchy-Schwarz. Sending $x\mapsto x + (b_j-b_{j-1})r$, we have 
\begin{align}
(\ref{twobt})=&
 \int \triangle^j f_j(x;B_ju) [\int\prod_{i=j+1}^{k} \triangle^j f(x-(b_i-b_{j})r;B_iu)\,dr]\,dx\,du\\\leq&\label{twob2}
 (|b_j-b_{j}|\cdots|b_j-b_0|)^{-\half}\|\triangle^j f_j\|_2 [\int|\int \prod_{i=j+1}^{k} \triangle^{j}f(x-(b_i-b_{j})r;B_iu)\,dr|^2\,dx\,du]^{\half}
\end{align}
and the $|b_j-b_n|^{-\half}$ factors may be dropped since $|b_j-b_i|\geq 1$ by assumption.

Expanding the square in the integral, calling $u_{j+1}$ the variable of integration of the second copy of the integral over $r$ and then changing variables $u_{j+1}\mapsto u_{j+1}+r$, this becomes

\begin{align}
(\ref{twob2})=&
\|\triangle^j f_j\|_2 [\int\int \prod_{i=j+1}^{k} \triangle^{j}f(x-(b_i-b_{j})r;B_iu)\\\cdot&
\triangle^{j}f(x-(b_i-b_{j})r-(b_i-b_{j})u_{j+1};B_iu)\,dr\,dx\,du\,du_{j+1}]^{\half}\\=&
  \|\triangle^j f_j\|_2 [\int\int   \prod_{i=j+1}^{k} \triangle^{j+1}f(x-(b_i-b_{j})r;B_iu,(b_i-b_j)u_{j+1})\,dr\,dx\,du\,du_{j+1}]^{\half}
\end{align}
which is what we claimed.

So by induction, we have

\begin{align}
 \int\int \prod_{i=0}^{k} f(x-b_ir)\\\leq&
 \left(\prod_{i=0}^{k-1}\|\triangle^if\|_{2}^{2^{-i}}\right) [\int \triangle^k f_k(x;B_ku)\,dx\,dr\,du]^{\frac{1}{2^k}}
\end{align}

and a change of variables gives the result. 

\end{proof}

\section{Estimates for the restricted operator}\label{ch:Differentiation}
Our goal is to show that $\mathcal{M}_n$ has $L^p\ra L^p$ mapping norm bounded by $C 2^{-n}$. As in \cite{Bourgain} and later in \cite{LabaDiff}, we use duality to recast this as follows. 

\begin{lemma}\label{thm:D1}
With $\mu$, $\mu_n$, $\mathcal{M}_n$ as above and $q=p'$, we have

\begin{align}\label{7.1}
&\|\mathcal{M}_n\|_{L^p\ra L^p} = \sup_{\|g\|_{q} = 1} (\int\sup_{t\in(1,2)} |\int g(x-ty) \, d\mu_n(x)|^{q}\,dy)^{\frac{1}{q}}	
\end{align}

\end{lemma}

\begin{proof}
 Duality allows us to write the $p$ norm of $\M f $ as a supremum over integrations of $\M f$ against functions $g\in L^{q}$ of norm $1$. If we expand $\M f$ as $\int f(x+t(x)y)\,d\mu(x)$, 
 where $t(x)\in [1,2]$ is chosen to approximate the supremum, this becomes

\begin{align}
 & \int g(x) \int f(x+t(x)y) \,d\mu_n(y) dx\\=&
\int f(x) g(x - t(x)y)\, dx d\mu_n(y) \\=&
\int f(x) \int g(x-t(x)y\,d\mu_n(y) \,dx
\end{align}

If we now apply Holder's inequality, we can bound this by  
\label{7222}
\begin{align}
\|f\|_p (\int |\int g(x - t(x)y)\,d\mu_n(y)|^{q}\,dx)^{\frac{1}{q}}
\end{align}
and taking the supremum over all $\|g\|_{q}=1$, then over all $\|f\|_p=1$ yields (\ref{7.1}).
\end{proof}

From here on out, we fix a function $t : \R\ra [1,2]$, and derive all bounds independent of the specific choice of $t$. In this manner we control the supremum over $t\in(1,2)$ in the definitions above.

By (\ref{7.1}), we look for a bound of the form 

\begin{align}
 (\int |\int g(x-t(x)y) \, d\mu_n(x)|^{q}\,dy)^{\frac{1}{q}} := \|\M_n^* g\|_{q} \leq C 2^{-\eta(p)n} \|g\|_{q} 
\end{align}

By interpolation, it is sufficient to check that (\ref{7222}) holds for $g$ the characteristic function of a set. This can be formalized as the following lemma, which we borrow wholesale from \cite{LabaDiff} (where it is Lemma 3.4)

\begin{lemma} \label{thm:interpolation}
Let $\M_n^{*}$ be as in (\ref{7222}) and $q_0 \geq 2$. Suppose that $\M_n^{*}$ 
obeys the restricted strong-type estimate 
\begin{equation}\label{7.7} \|\M_n^{*} \mathbf 1_{\Omega} \|_{q_0} \leq C 2^{-k \eta_0}|\Omega|^{\frac{q_0-1}{q_0}} \quad \text{ for all sets } \Omega \subseteq [0,1]   \end{equation}  
with some $\eta_0>0$.
Then for any $q < q_0$ there is an $\eta(q)>0$ such that $\M_n^*$ is bounded from $L^q[0,1]$ to $L^{q'}[-4,0]$ with operator norm at most $C 2^{-n\eta(q)}$. 
\end{lemma} 
  
\begin{proof}
Trivially $\|\M_n^{*}\|_{L^1\ra L^1} < \infty$. Interpolation of restricted weak-type endpoint bounds (Section 4, Theorem 5.5, \cite{BS}) guarantees that $\M_n^{*}$ is bounded from $L^p \rightarrow L^q$
for all $(p,q)$ satisfying $p' = q_0/\theta$ and $q' = q_0/(\theta (q_0-1))$, $0 < \theta < 1$ independent of $n$. 

We again interpolate, in the form of Holder's inequality, to obtain decay in $n$

\begin{align*} 
\|\M_n^{*}\mathbf 1_{\Omega}\|_q &\leq \|\M_n^{*}\mathbf 1_{\Omega}\|_{q_0}^{\theta} \|\M_n^{*}\mathbf 1_{\Omega}\|_1^{1 - \theta} \leq C2^{-n\eta_0 \theta} |\Omega|^{\frac{1}{p}}.
\end{align*} 

Together, this gives that the weak-type $(p,q)$ norm of $\M_n^{*}$ is bounded by $C'2^{-n\eta_0 \theta}$ after an application of Theorem 5.3 of \cite[Section 4]{BS}).

To gain the strong bounds we desire, we now apply Marcinkiewicz interpolation to any such $(p_1,q_1)$, $(p_2,q_2)$ to obtain our conclusion.

\end{proof}

Now choose $q_0=k$. Then expanding (\ref{7.7}), we seek bounds for the expression

\begin{align}
\|\mathcal{M}_n^{*}1_{\Omega}\|_{k}^k=& \int \prod_{i=1}^{k}\int 1_{\Omega}(x_i-t(x_i)y) \mu_n(x_i)\,dx_1\,\cdots dx_k\,dy\\\label{7.8}=&
\int_{\Omega\times\cdots\times\Omega}\int \prod_{i=1}^{k}  \mu_n(x_i + t(x_i)y) \,dy\,dx_1\,\cdots dx_k
\end{align}

A portion of this integral may be controlled without assuming anything - this is the ``internal tangency'' portion of the argument, and we begin it with the following lemma.

For $x\in\R^d$, set $|x|_{\min} = \min_{1\leq i \leq d} |x_i|$. 

\begin{lemma}\label{thm:internals}
 Let $\Omega\subset\R^d$ be any set of positive measure, and $A = \big\{ x\in ([0,1]^{d})^k : |x_i|_{\min}\leq \delta \text{ or } |x_i - x_{j}|_{\min} \leq \delta, 1\leq i\neq j\leq k\big\}$ 
 where $\delta>0$ is some parameter. 
Then there is a finite constant $C(k,d)>0$ for which
\begin{align}
|\Omega^k\cap A| \leq  C(k,d) \delta |\Omega|^{k-1}
\end{align}
\end{lemma}
\begin{proof}
 Let $x_i\in\R^d$ have components $x_i^j$ so that $x_i=(x_i^1,\dots,x_i^d)$. Write $A$ as the union of $\delta$ thickenings of the $kd$ $1$-dimensional subspaces 
 $S_{0,i;j}:=\{x\in [0,1]^{kd}: x_i^j=0\}$, $S_{i,i';j}:=\{x\in[0,1]^{kd} : x_{i}^j=x_{i'}^j\}$, $i=1,\dots,k-1$, $j=1,\dots,d$. Certainly, letting $(E)_{\delta}$ denote a $\delta$-thickening of the set $E$, we have 
\begin{align*}
 |\Omega^k\cap (S_{0,i;j})_{\delta}| \leq |\left\{  (x_1,\dots,x_k)\in[0,1]^{kd} : |x_i^j|<\delta, x_{i'}\in\Omega,i'\neq i, i=1,\dots,k  \right\}| = C_d \delta |\Omega|^{k-1}
\end{align*}

Similarly, $|\Omega^k\cap (S_{i,i';j})_{\delta}| \leq |\left\{  (x_1,\dots,x_k) : x_n\in \Omega, n\neq i', x_i^j\in B(x_{i'}^j,\delta) \right\}| = C_d \delta|\Omega|^{k-1}$.
\end{proof}

Now (\ref{7.8}) may be controlled through an application of this lemma together with our higher-order methods.

\begin{thm}\label{thm:d11}
 There is a constant $C$ depending only on the $k-1$-st order Fourier dimension $\beta$ of $\mu$, the constants $C_H$ and $\alpha$ in (\ref{ballcondition}), and $d$ and $k$
 so that (\ref{7.8}) is bounded by $C 2^{-n\eta} |\Omega|^{k-1}$.
\end{thm}

\begin{proof}

To begin, we may assume $\beta_0<d$ large enough that $r_k(\beta)>0$ for all $\beta>\beta_0$, which is possible by Proposition \ref{thm:rk}, and next suppose $\alpha<d$ large enough
that $\frac{r_k(\beta_0)}{2^k}>k(d-\alpha)$, which implies $\frac{r_k(\beta)}{2^k}-k(d-\alpha)>0$ since by Proposition \ref{thm:rk}, $r_k$ is increasing.

Let $\delta=2^{-nk\eps}$ where $\eps>d-\alpha>0$, and 
\[A = \big\{ x\in [0,1]^k : |x_i|_{\min}\leq 2^{-nk\eps} \text{ or } |x_i - x_{j}|_{\min} \leq 2^{-nk\eps} , 1\leq i,j\leq k\big\}\]

Set $\Omega_0 = \left\{x\in\Omega ' : |x_i|_{\min}\leq 2^{-nk\eps} \text{ or } |x_i - x_{j}|_{\min} \leq 2^{-nk\eps}\right\} = \Omega^k\cap A$, and note that the integral (\ref{7.8}) is

\begin{align}\label{Ommeggga}
 &\int_{\Omega_0} \int \prod_{i=1}^{k} \mu_n(x_i-t(x_i)y) \,dx\,dy\\\label{omeggga}&+
\int_{\Omega^k\setminus\Omega_0} \int \prod_{i=1}^{k} \mu_n(x_i-t(x_i)y) \,dx\,dy
\end{align}

We first estimate (\ref{Ommeggga}).
 
Since $\|\mu_n\|_{\infty} = \|\phi_n\ast\mu\|_{\infty} \leq C 2^{n(d-\alpha)}$ where $\alpha$ is the Hausdorff dimension of $\mu$ (this is a standard consequence of the uncertainty principle, see e.g. \cite{Erdogan}, Section 4), we have (after a change of variables on $y$)

\begin{align}\label{webegin}
 \int_{\Omega_0} \int_{[0,1]^{d}} \prod_{i=1}^{k} \mu_n(x_i - t(x_i) y) \,dy\,dx \leq C 2^{nk(d-\alpha)} \int 1_{\Omega_0}1_{A} \leq C 2^{nk(d-\alpha)} |\Omega^k\cap A|
\end{align}

By Lemma \ref{thm:internals}, $|\Omega^k\cap A| \leq C(d,k) 2^{-nk\eps}|\Omega|^{k-1}$, so that
\begin{align}\label{wefinshing}
 (\ref{webegin})\leq C 2^{-nk(\eps-(d-\alpha))}|\Omega|^{k-1} : = C 2^{-n\eta_0}|\Omega|^{k-1}
\end{align}
 which decays as we chose $\eps>d-\alpha$.

This leaves us with a need to estimate the integral (\ref{omeggga})

\begin{align*}
 \int_{\Omega^k\setminus\Omega_0} \int \prod_{i=1}^{k} \mu_n(x_i-t(x_i)y) \,dx\,dy
\end{align*}

 Since the expression is symmetric in each copy of $\Omega$, we may write 
 \begin{align}
  \Omega ' := \left\{ x\in\Omega^k : |x_1|_{\min}\leq| x_2|_{\min}\leq\cdots\leq |x_k|_{\min}\right\}
 \end{align}
 and consider an integral over this region instead, since (\ref{7.8}) will be its constant multiple.

Define $f$ so that $f(y-t(x)^{-1}x) = \mu_{n}(t(x)x - y)$. 

Since $|x_1|_{\min} > 2^{-nk\eps}$, $|x_{i+1}-x_i|_{\min}>2^{-nk\eps}$, and $|x_i|_{\infty}\in[0,1]$, we can define $b_i=b_i(x_i)\in\R^d$ with $|b_i(x_i)|_{\min}\in[1,C2^{nk\eps)}),|b_i-b_{i'}|_{\min}\in[1,C2^{nk\eps)})$ 
for which $x_i = b_i(x_i)x_1$ (recall for vectors $v$ and $u$, $vu := (v_1u_1,\dots,v_du_d)$). 

Then setting $B=\prod_{i=2}^{k-1}[i,2^{nk\eps}]\setminus\left\{|b_{i+1}-b_i|_{\min}< 1\right\}$, after a change of variables we can express (\ref{omeggga}) as

\begin{align}
 \int_{B} \int_{[0,1]^d} \int 1_{\Omega '\setminus\Omega_0} (x_1,b_2 x_1,\dots,b_k x_1) \prod_{i=1}^{k} f(y - t_i b_i x_1) |x_1|^{k-1}\,dy\,dx_1\,db
\end{align}
where $t_i=t(x_i)^{-1}$. This we may bound through an application of Holder's inequality as

\begin{align}
 &[\sup_{b\in B} \int_{[0,1]^d} \int 1_{\Omega '\setminus\Omega_0} (x_1,b_2x_1,\dots,b_kx_1) \prod_{i=1}^{k} f(y - t_i b_i x_1) \,dy\,dx_1]\label{771}\\&\cdot
[\sup_{x_1\in[0,1]^d} \int_{B} |x_1|^{k-1}  1_{\Omega '\setminus\Omega_0} (x_1,b_2x_1,\dots,b_kx_1) \,db]\label{772}
\end{align}

After changing variables back, we bound (\ref{772}) by $|\Omega'\setminus\Omega_0|\leq|\Omega|^{k-1}$.

Bounds for (\ref{771}) on the other hand come from Lemma \ref{thm:transverse} applied to 
\[ [\sup_{b\in B} \int_{[0,1]^d} \int[0,1]^d \prod_{i=0}^{k} F_i(x - (t_i b_i)r) \,dx\,dr]\]
with $x=y, r=x_1$, $F_0\equiv 1$, and $F_{i\geq 1} = f$, which majorizes (\ref{772}).


The higher-order Fourier dimension assumption together with Proposition \ref{thm:rk} tells us that
\begin{align}\label{216}
 \|f\|_{U^k} = \|\mu_{n+1}-\mu_n\|_{U^k} \lesssim 2^{-\frac{r_k}{2^k} n}
\end{align}
and the ball condition (\ref{ballcondition}) tells us that 
\begin{align}\label{215}
 \|f\|_{\infty} \lesssim 2^{(d-\alpha)n}
\end{align}
Putting these bounds together via Lemma \ref{thm:transverse},
we've shown that

\begin{align}
&(\ref{771})\leq (\ref{215})^{k-1}\cdot(\ref{216}) = 2^{-(\frac{r_k}{2^k}-k(d-\alpha))n}
\end{align}

Combining the estimates on (\ref{771}) and (\ref{772}),

\begin{align}
&(\ref{omeggga})\lesssim 2^{-(\frac{r_k}{2^k}-k(d-\alpha))n} |\Omega|^{k-1}\\
:=& 2^{-n\eta_1} |\Omega|^{k-1}
\end{align}

We conclude from our initial assumptions on the size of $\beta_0$ and $\alpha$ that   $\eta_1>0$ for all $\beta>\beta_0$.

Then by this bound on (\ref{omeggga}) and the bound (\ref{wefinshing}) on (\ref{Ommeggga}), for all $\beta>\beta_0$
there is a positive $\eta=\min(\eta_0,\eta_1)=\eta(\beta)$ so that (\ref{7.8}) is bounded by a constant multiple of $ 2^{-n\eta}|\Omega|^{k-1}$.
\end{proof}

So, if $\mu$ has a $k-1$-st order Fourier dimension close enough to $d$, Theorem \ref{thm:d11} together with Lemma \ref{thm:interpolation} tell us that $\mathcal{M}_n^*$ is bounded from  $L^k([0,1])$ to $L^{\frac{k}{k-1}}([0,1])$ with a bound decaying exponentially in $n$. 
By Lemma \ref{thm:D1} and the discussion following it, this means that $\mathcal{M}_n$ is bounded from, say, $L^{\frac{k}{k-1}}([0,1]^d)$ to $L^{\frac{k}{k-1}}([-4,4]^d)$ with a bound that also decays exponentially in $n$. Summing these up, we obtain a bound on $\mathcal{M}$ from $L^{\frac{k}{k-1}}([0,1]^d)$ to $L^{\frac{k}{k-1}}([-4,4]^d)$. Putting this together with Lemma \ref{thm:d6} below, we conclude

\begin{thm}\label{thm:sooomany}
 There is a constant $C>0$ depending on the $k-1$-st order Fourier decay rate of $\mu$, so that for any $p> {\frac{k}{k-1}}$, $\sup_{t\in[1,2]} \int f(x-ty)\,d\mu(y)$ has an $L^p$ norm less than $C(p)\|f\|_{L^p}$.
\end{thm}

So we proceed to show that the restriction to support in $[0,1]$ may be dropped. This is the result of the final lemma, a standard application of disjointness of support.

\begin{lemma}\label{thm:d6}Suppose that for $f$ supported in $[0,1]^d$, $\|\mathcal{M} f\|_p\leq C \|f\|_p$. Then in fact, ${\mathcal{M}}$ is bounded from $L^p(\R^d)$ to $L^p(\R^d)$.

\end{lemma}

\begin{proof}
Let $(J_i)$ denote an enumeration of the dyadic cubes of sidelength $1$ in $\R^d$, and write $f =\sum f$, $\supp(f) \subset J_i$.
Then using the subadditivity of ${\mathcal{M}}$, and the fact that $\supp{\mathcal{M}}f_i$ is at most increased by $5^d$ in size
(since $\supp(f_i\ast\mu)\subset\supp(f_i)+\supp(\mu)$), we have that at most $5^d$ of the $\supp{\mathcal{M}}f_i$ overlap, hence
\begin{align}
 \|\mathcal{M}f\|_p^p \leq \|\sum_i \mathcal{M}f_i\|_p^p\leq\\&
\leq 5^d \sum \|{\mathcal{M}}f_i\|_p^p
\end{align}

Applying now the hypothesized bounds and condensing the sum into $\|f\|_p^p$, we obtain the claim.
\end{proof}

\section{Scaling and the unrestricted operator}
\label{ch:scaling}

In this section we derive the lemmas necessary to reduce boundedness of the maximal operator $\tilde{\mathcal{M}}$ defined via a supremum over the full range of scales $t>0$, to the boundedness of the operator $\mathcal{M} : f\mapsto \sup_{t\in[1,2]} f(x+ty)\,d\mu(y)$. The main result of the section is Lemma \ref{thm:scale5}, which concludes the boundnedness of $\tilde{\mathcal{M}}$ on functions of compact support. Lemma \ref{thm:scale6} then concludes $L^p(\R^d)$ bounds from this, completing the proof.

Let $\mathbb{E}_s$ denote conditional expectation with respect to the dyadic $\sigma$-algebra generated by dyadic cubes of length $2^{-s}$.
 Let $f^*$ denote the Hardy-Littlewood maximal function for $f$.

\begin{lemma}\label{thm:scale1}
For any $s\in\Z$, we have $|\int \mathbb{E}_s f(x+ry) \,d\mu(y) |\leq 5^d f^{*}(x)$.
\end{lemma}
\begin{proof}

We may as well suppose $f$ a positive function.

 Let $A$ denote the transformation $y\mapsto x+ry$, and set $\mu_A = A_{*}\mu$.

We divide $A (\supp(\mu)) \subset A[0,1]^d$ into at most $5^d$ subintervals $J_i$, $i=1,\dots,5^d$ of length $2^{-s}$. Then for any $x_i\in J_i$, say, 

\begin{align}
 &\int \mathbb{E}_s f \,d\mu_A =\\&
\sum_i \int_{J_i} \mathbb{E}_s f \,d\mu_A=\\&
\sum_i \mathbb{E}_s f(x_i) \mu_A(J_i)=\\&
\sum_i \mu_A (J_i) \mathbb{E}_s f(x_i)
\end{align}

Now
\begin{align}\mathbb{E}_s f(x_i) = \frac{1}{|J_i|} \int_{J_i} f \leq \frac{5^d}{5^d J_i} \int_{5 B(x_i,2^{-s})} f\leq 5^d f^{*}(x)\end{align}
and $\sum_i \mu_A(J_i) = 1$, whence $\int \mathbb{E}_s f\,d\mu_A \leq 5^d f^{*}(x)$.
\end{proof}

\begin{lemma}\label{thm:scale2}
 Suppose that $k\leq c s^{\half}$, $\mu_k:= \phi_k\ast\mu$ for some $\phi_k$ with $\|\phi_k'\|_{\infty}\leq 2^{2k}$, say $\phi_k (x) = 2^k\phi(2^k x)$ where $\|\phi'\|_{\infty}\leq 1$. Then for $|y_1-y_2|\leq 2^{-s}$ and $t(x)\in(1,2)$
\begin{align}
\|\mu_k(\frac{y_1-x}{t(x)}-\cdot) - \mu_k(\frac{y_2-x}{t(x)}-\cdot) \|_{\infty} \leq 2^{-(s-2cs^{\half})}
\end{align}

\end{lemma}
\begin{proof}
 If $|x-y|\leq 2^{-s}$, then
\begin{align}|\phi_k(x)-\phi_k(y)| < 2^{2k}|x-y| < 2^{-(s-2ck)} \leq 2^{-(s-2s^{\half})}\end{align}
by our assumption on $k$.

Then 
\begin{align}
 &|\mu_k(\frac{y_1-x}{t(x)})-\mu_k(\frac{y_2-x}{t(x)})|\leq\\&
\int |\phi_k(\frac{y_1-x}{t(x)}-y)-\phi_k(\frac{y_2-x}{t(x)}-y)|\,d\mu(y)\leq\\&
2^{-(s-2cs^{\half})}
\end{align}
\end{proof}

\begin{lemma}\label{thm:scale3}
 Suppose that $\mathbb{E}_s f= 0$ and that $ k\leq c s^{\half}$, with $\mu_k$ defined as in Lemma \ref{thm:scale2} above. Then 
\begin{align}
 \|\mathcal{M} f\|_p\leq 2^{-(s-2cs^{\half})} \|f\|_p
\end{align}

\end{lemma}
\begin{proof}
Breaking into subintervals $I_j$ of center $y_j$ and length $2^{-s}$, we have $\sum_j \mathbb{E}_s(\mu_k(\frac{y_j-x}{t(x)})) \int_{I_j} f = 0$. Consequently

\begin{align}
 &\int |\mathcal{M}f(x)|^p \,dx =\\&
\int |\int f(y)\mu_k(\frac{y-x}{t(x)}) dy|^p \,dx =\\&
\int |\sum_j \int_{I_j} f(y)[\mu_k(\frac{y-x}{t(x)}) -\mu_k(\frac{y_j-x}{t(x)}) ]\,dy|^p \,dx 
\end{align}

set $\bar{\mu_k}(\frac{y-x}{t(x)}) = \mu_k(\frac{y_j-x}{t(x)})$ for $y\in I_j$. Then this becomes 
\begin{align}
 \int | \int f(y)[\mu_k(\frac{y-x}{t(x)}) -\bar{\mu_k}(\frac{y-x}{t(x)})]\,dy|^p \,dx 
\end{align}

An application of Holder, followed by the bound from Lemma \ref{thm:scale2}, furnishes us with

\begin{align}
 &\int | \int f(y)[\mu_k(\frac{y-x}{t(x)}) -\bar{\mu_k}(\frac{y-x}{t(x)})]\,dy|^p \,dx \leq\\&
 \int | ( \int |f(y)|^p)^{\frac{1}{p}} (\int |\mu_k(\frac{y-x}{t(x)}) -\bar{\mu_k}(\frac{y-x}{t(x)})|^q\,dy)^{\frac{1}{q}}|^p \,dx \leq\\& 
\|f\|_p^p 2^{-p(s-2cs^{\half})} \,dx  
\end{align}

\end{proof}

This next lemma is what allows us to leverage control over restricted $\mathcal{M}_k$ to some measure of control over the unrestricted $\tilde{\mathcal{M}}$.

\begin{lemma}\label{thm:scale4}
 Suppose that $\mathbb{E}_{\nu+s}(f)=0$. Then if $\|\mathcal{M}_k f\|_p\lesssim 2^{-k}\|f\|_p$, we have
\begin{align}\label{ten}
 \|\sup_{t\approx 2^{-\nu}}| \int f(x+ty)\,d\mu(y)|\|_p \leq c 2^{-a(p)s}\|f\|_p
\end{align}

\end{lemma}

\begin{proof}
By rescaling, we may assume that $\nu=0$.

Then we compute
\begin{align}
 &\|\sup_{t\approx 1} | \int f(x+ty)\,d\mu(y)|\|_p\leq\\&
c \sum_k \| \int f(x+ty)\,d\mu_k(y)\|_p\lesssim\\&
\sum_{k < c\sqrt{s}} \| \int f(x+ty)\,d\mu_k(y)\|_p+\sum_{k \geq c\sqrt{s}} \| \int f(x+ty)\,d\mu_k(y)\|
\end{align}
 
We may now apply Lemma \ref{thm:scale3} to the left term, bounding it by $\sum_{k < c\sqrt{s}} 2^{s-2cs^{\half}}\| f\|_p\leq 2^{s-3cs^{\half}}\|f\|_p$. 

At the same time, our bounds on $\mathcal{M}_k$ give us that the right term is $\leq\ \sum_{k\geq c\sqrt{s}} 2^{-ck}\|f\|_p\leq 2^{-cs^{\half}}\|f\|_p$

\end{proof}

The following lemma gives us the bound we will ultimately tweak to complete the argument.

\begin{lemma}\label{thm:scale5}
 Set $\triangle_k f = \mathbb{E}_{k+1}f-\mathbb{E}_{k}f$. Suppose that $\supp(f)\subset[0,1]$. Then 
\begin{align}
 \|\tilde{\mathcal{M}}f\|_p \leq C \|f\|_p
\end{align}

\end{lemma}

\begin{proof}
 By Lemma \ref{thm:scale1} and the Hardy-Littlewood differentiation theorem, $\|\mathbb{E}_{\nu} f\|_p\leq c ||f||_p$. So writing $f = \mathbb{E}_{\nu} f + \sum_{k\geq\nu} \int \triangle_k f(x+ty)\,d\mu(y)$
 we are led to bound 
 \[\|\sup_{\nu\geq 0}\sup_{2^{-\nu}\leq t\leq 2^{-\nu+1}} |\sum_{k\geq\nu} \int \triangle_kf(x+ty)\,d\mu_k(y)\|_p\]

We have that 

\begin{align} 
&\|\sup_{\nu\geq 0}\sup_{2^{-\nu}\leq t\leq 2^{-\nu+1}} \left|\sum_{k\geq\nu} \int \triangle_kf(x+ty)\,d\mu_k(y)\right|\|_p\leq\\&
\|\left(\sum_{\nu\geq 0}\left[\sum_{s=0}^{\infty} \sup_{t\approx 2^{-\nu}} | \int \triangle_{\nu+s}f(x+ty)\,d\mu_k(y)|\right]^2\right)^\frac{1}{2}\|_p\leq\\&
\|\sum_{s=0}^{\infty} \left[ \sum_{\nu\geq0}  \sup_{t\approx 2^{-\nu}} | \int \triangle_{\nu+s}f(x+ty)\,d\mu_k(y)|^2\right]^{\frac{1}{2}}\|_p\leq\\&
\sum_{s=0}^{\infty} \left( \sum_{\nu\geq0} \|\sup_{t\approx 2^{-\nu}} | \int \triangle_{\nu+s}f(x+ty)\,d\mu_k(y)|\|_p^2\right)^{\frac{1}{2}}
\end{align}
where the last two lines follow via Minkowski's inequality.

Applying Lemma \ref{thm:scale4} to $\triangle_{\nu+s} f$, we have that this is bounded by a multiple of
\begin{align}
 \sum_{s=0}^{\infty} \left(\sum_{\nu\geq0} 2^{-a(p)s}\|\triangle_{\nu+s} f\|_p^2\right)^{\frac{1}{2}}
\end{align}

One last application of Minkwoski's useful inequality, and this becomes
\begin{align}
 &\sum_{s=0}^{\infty} 2^{-\frac{a(p)}{2} s} \|\left(\sum_{\nu\geq0} |\triangle_{\nu+s} f|^2\right)^{\frac{1}{2}}\|_p\lesssim&\\
\sum_{s=0}^{\infty} 2^{-\frac{a(p)}{2} s}\|f\|_p\lesssim \|f\|_p
\end{align}
where the last line follows from Littlewood-Paley theory.

\end{proof}

We now show that the restriction to support in $[0,1]^d$ may be dropped. Though we include it here, it is identical to Lemma \ref{thm:d6} of the previous section. 

\begin{lemma}\label{thm:scale6}Suppose that for $f$ supported in $[0,1]^d$, $\|\tilde{\mathcal{M}} f\|_p\leq C \|f\|_p$. Then in fact, $\tilde{\mathcal{M}}$ is bounded from $L^p(\R^d)$ to $L^p(\R^d)$.

\end{lemma}

\begin{proof}
Let again $(J_i)$ denote an enumeration of the dyadic intervals of sidelength $1$, and write $f =\sum f$, $\supp(f) \subset J_i$. Then using the subadditivity of $\tilde{\mathcal{M}}$, and the fact that $\supp\tilde{\mathcal{M}}f$ 
increases in size by at most $5^d$ (since $\supp(f\ast\mu)\subset\supp(f)+\supp(\mu)$), we have
\begin{align}
 \|\tilde{\mathcal{M}} f\|_p^p \leq \|\sum_i \tilde{\mathcal{M}} f\|_p^p\leq\\&
\leq 5^d \sum \|\tilde{\mathcal{M}}f\|_p^p
\end{align}

Applying now the hypothesized bounds and condensing the sum into $\|f\|_p^p$, we're done.
\end{proof}


\bibliographystyle{plainnat}

\bibliography{biblio.bib}

%
%
%

\vskip0.5in

\noindent Marc Carnovale
\\ The Ohio State University \\ 231 W. 18th Ave\\ Columbus Oh, 43201 United States\\ {\em{Email: carnovale.2@osu.edu}}

\end{document}